\documentclass[11pt,a4paper]{article}

\usepackage[T1]{fontenc}
\usepackage{newtxtext}
\usepackage{mathtools,amsthm}
\usepackage{newtxmath}
\usepackage{microtype}
\usepackage{mleftright}\mleftright

\usepackage[margin=2.7cm]{geometry}

\usepackage{enumitem}
\setlist[enumerate]{itemsep=2pt,topsep=4pt}

\usepackage{xcolor}
\definecolor{linkblue}{RGB}{35,65,110}

\usepackage[
    colorlinks=true,
    linkcolor=linkblue,
    citecolor=black,
    urlcolor=linkblue
]{hyperref}
\hypersetup{
    pdftitle={Dual medial axes in finite-dimensional normed spaces},
    pdfauthor={Arnaud Durand, Xiangyu Liang, Sicheng Zhang}
}

\allowdisplaybreaks
\numberwithin{equation}{section}

\makeatletter
\renewenvironment{proof}[1][\proofname]{\par
    \pushQED{\qed}%
    \normalfont
    \topsep6\p@\@plus6\p@\relax
    \trivlist
    \item[\hskip\labelsep\bfseries #1\@addpunct{.}]\ignorespaces
}{%
    \popQED\endtrivlist\@endpefalse
}
\makeatother

\newtheorem{theorem}{Theorem}[section]
\newtheorem{proposition}[theorem]{Proposition}
\newtheorem{lemma}[theorem]{Lemma}
\newtheorem{corollary}[theorem]{Corollary}

\theoremstyle{definition}
\newtheorem{definition}{Definition}[section]

\theoremstyle{remark}
\newtheorem{question}{Question}
\newtheorem*{remark}{Remark}

\newcommand{\R}{\mathbb{R}}

\newcommand{\Sp}{\mathbb{S}}

\newcommand{\cM}{\mathcal{M}}

\newcommand{\aff}{\operatorname{aff}}

\newcommand{\spann}{\operatorname{span}}
\newcommand{\supp}{\operatorname{supp}}
\newcommand{\conv}{\operatorname{conv}}
\newcommand{\interior}{\operatorname{int}}

\DeclarePairedDelimiter{\norm}{\lVert}{\rVert}
\DeclarePairedDelimiter{\enorm}{\lvert}{\rvert}
\DeclarePairedDelimiterX{\dualnorm}[1]{\lVert}{\rVert_{*}}{#1}

\newcommand{\doi}[1]{\href{https://doi.org/#1}{doi:#1}}

\title{Dual medial axes in finite-dimensional normed spaces}

\author{
    Arnaud Durand\thanks{Université Paris-Saclay, CNRS, Laboratoire de mathématiques d'Orsay, 91405, Orsay, France\\ \href{mailto:arnaud.durand@universite-paris-saclay.fr}{\texttt{arnaud.durand@universite-paris-saclay.fr}}}
    \and
    Xiangyu Liang\thanks{School of Mathematical Sciences, Beihang University, Beijing 102206, China\\ \href{mailto:maizeliang@gmail.com}{\texttt{maizeliang@gmail.com}}}
    \and
    Sicheng Zhang\thanks{School of Mathematical Sciences, Peking University, Beijing 100871, China\\ \href{mailto:zhangsicheng113@gmail.com}{\texttt{zhangsicheng113@gmail.com}}}
}

\date{}

\begin{document}

\maketitle

\begin{abstract}
    We study medial axes of nonempty closed sets in finite-dimensional real normed spaces whose norms are $C^1$ away from the origin. Using the differential of the norm, we define the dual $k$-medial axis by requiring the contact covectors associated with nearest points to have affine dimension at least $k-1$. In dimension $n$, we prove that this set is $(n-k+1)$-rectifiable for every $1\le k\le n+1$, and hence has Hausdorff dimension at most $n-k+1$. We give a direct geometric proof and characterize the dual medial axes through the superdifferential of the squared distance function, obtaining an alternative proof from a theorem of Alberti, Ambrosio, and Cannarsa on singularities of semiconvex functions. Under the additional assumption of strict convexity, the ordinary and dual medial axes coincide for $k=2,3$, yielding the corresponding rectifiability results for ordinary medial axes. For $k=4$, however, the ordinary rectifiability conclusion can fail even for norms arbitrarily close to Euclidean: we construct norms on $\R^3$ with smooth positively curved unit spheres arbitrarily $C^\infty$-close to the Euclidean sphere, together with four-point sets whose ordinary $4$-medial axes contain nondegenerate line segments.
\end{abstract}



\section{Introduction}\label{sec:intro}

The medial axis was introduced by Blum~\cite{Blum1967} in 1967 as a biologically motivated tool for shape representation and recognition. His grassfire transform describes the medial axis through the meeting of inward-propagating fronts. Medial axes and medial representations have since found applications in shape analysis, image processing, computer vision, and geometric modeling~\cite{SiddiqiPizer2008}. Their topology has also been studied extensively. Lieutier~\cite{Lieutier2004} proved that every bounded open subset of $\R^n$ has the same homotopy type as its medial axis, defined by the existence of at least two nearest boundary points. In the planar setting, Choi et al.~\cite{Choi1997} established a finite graph structure under suitable piecewise real-analytic boundary assumptions, using the maximal-inscribed-disk convention.

We are concerned with the size and rectifiability of higher-order medial axes. Following Erd\H{o}s~\cite{Erdos1945} and Liang~\cite{Liang2022}, let $E\subseteq\R^n$ be nonempty and closed. For $1\le k\le n+1$, we define its $k$-medial axis $M_k(E)$ to be the set of points outside $E$ whose nearest-point set contains $k$ affinely independent points. Equivalently, the affine dimension of the nearest-point set is at least $k-1$. In particular, $M_1(E)=\R^n\setminus E$, while $M_2(E)$ consists of the points having more than one nearest point in $E$.

Erd\H{o}s~\cite{Erdos1945,Erdos1946} studied the size of these sets, proving in particular that $M_2(E)$ has Lebesgue measure zero and that $M_{n+1}(E)$ is countable. More generally, for $2\le k\le n$, he showed that $M_k(E)$ is contained in a countable union of surfaces of finite $(n-k+1)$-dimensional measure, and hence that
\[
    \dim M_k(E)\le n-k+1,
\]
where $\dim$ denotes Hausdorff dimension. His argument uses a geometric criterion formulated in terms of contingents. More recently, Liang~\cite{Liang2022} established the following explicit rectifiability statement, which provides the starting point for our investigation.

\begin{theorem}[Liang~\cite{Liang2022}]\label{thm:liang}
For every nonempty closed set $E$ of the Euclidean space $\R^n$ and every integer $1\le k\le n+1$, the $k$-medial axis $M_k(E)$ of $E$ is $(n-k+1)$-rectifiable. In particular,
\[
    \dim M_k(E)\le n-k+1.
\]
\end{theorem}

The goal of this paper is to determine how far this conclusion remains valid in finite-dimensional normed spaces. A direct extension to arbitrary norms fails. Although bi-Lipschitz equivalence of norms preserves the rectifiability of a fixed set, it does not preserve nearest points. For example, equip $\R^2$ with the infinity norm $\norm{\cdot}_\infty$ and let $E=\{(-1,0),(1,0)\}$. If $\lvert y\rvert>1+\lvert x\rvert$, then both points of $E$ are nearest to $(x,y)$, since their distance from $(x,y)$ is $\lvert y\rvert$. Thus
\[
    \{(x,y)\in\R^2:\lvert y\rvert>1+\lvert x\rvert\}\subseteq M_2(E),
\]
so $M_2(E)$ has nonempty interior and is not $1$-rectifiable. This example shows that extending Theorem~\ref{thm:liang} beyond the Euclidean setting requires additional assumptions on the norm or a modification of the notion of higher-order medial axis.

Nearest-point geometry beyond the Euclidean setting has been studied from several perspectives. Convex-distance Voronoi diagrams provide a computational framework for finite sets of sites~\cite{ChewDrysdale1985}, while bisectors, circumspheres, and related proximity structures are central topics in Minkowski geometry~\cite{Icking1995,Martini2004}. For arbitrary closed sets, Kolasiński and Santilli~\cite{Kolasinski2023} studied anisotropic distance functions and cut loci under the assumption that the norm is $C^2$ away from the origin and uniformly convex in the sense of their Definition~2.8. Here we study the first-order geometry of equidistance conditions under $C^1$ regularity and its role in the rectifiability of higher-order medial axes.

Let $X$ be an $n$-dimensional real normed space, with $n\ge1$, whose norm $\norm{\cdot}$ is $C^1$ on $X\setminus\{0\}$, and let $D$ denote the Fréchet derivative. The dual space and dual norm are denoted by $X^*$ and $\dualnorm{\cdot}$. The unit spheres are
\[
    \Sp_X=\{u\in X:\norm{u}=1\}\quad\text{and}\quad\Sp_{X^*}=\{f\in X^*:\dualnorm{f}=1\}.
\]
Define the duality map between the unit spheres by
\[
    J:\Sp_X\longrightarrow\Sp_{X^*},
    \qquad
    J(u)=D\norm{\cdot}(u).
\]
For each $u\in\Sp_X$, the covector $J(u)$ is the unique norming functional of $u$; this and the basic identities satisfied by $J$ are established in Lemma~\ref{lem:duality-identities}.

Let $E\subseteq X$ be nonempty and closed, and let $d_E$ denote the distance to $E$, defined for $x\in X$ by
\[
    d_E(x)=\inf_{y\in E}\norm{x-y}.
\]
If $x\notin E$ and $y\in E$ is a nearest point to $x$, we call $u=(y-x)/d_E(x)$ a contact vector and $J(u)$ the corresponding contact covector. At a fixed center $x$, the contact vectors are obtained from the nearest points by a translation and a nonzero scaling, so the two sets have the same affine dimension. Thus $x\in M_k(E)$ precisely when its contact vectors have affine dimension at least $k-1$. To understand the local geometry of medial axes, we first linearize the equidistance conditions for fixed contact points. The expansion
\[
    \norm{y-(x+h)}=d_E(x)-J(u)(h)+o(\norm{h})
\]
shows that the contact covector determines the first-order change in the distance to $y$. For nearest points $y_1,\ldots,y_q$ with contact vectors $u_1,\ldots,u_q$, preserving equidistance to first order therefore imposes the conditions $(J(u_i)-J(u_1))(h)=0$ for $i=2,\ldots,q$. The intersection of these kernels has codimension equal to the affine dimension of $\{J(u_1),\ldots,J(u_q)\}$, which thus counts the independent first-order equidistance constraints on the center.

This motivates the definition of the dual $k$-medial axis $M_k^J(E)$ as the set of points outside $E$ whose contact covectors have affine dimension at least $k-1$. The dual medial axes thus classify nearest-point configurations by the rank of their linearized equidistance conditions. This rank is intrinsic to the norm and the nearest-point configuration. It also admits a characterization through the distance function itself. Indeed, Proposition~\ref{prop:gEconv} shows that $M_k^J(E)$ is precisely the set of points outside $E$ where the superdifferential of $d_E^2/2$ has dimension at least $k-1$.

For norms induced by an inner product, the first-order equidistance conditions characterize exact equidistance to the fixed contact points: in the differences of squared distances, the quadratic terms cancel, leaving precisely the linearized equations. This agrees with the identification of contact vectors and contact covectors through the Riesz isomorphism, whose restriction to the unit sphere is $J$. Since this isomorphism preserves affine dimension, the ordinary and dual definitions coincide: $M_k^J(E)=M_k(E)$. For general $C^1$ norms, the affine dimensions of the contact vectors and contact covectors may differ, and so may the corresponding medial axes.

At a point of $M_k^J(E)$, the linearized equidistance conditions leave a space of displacements of dimension at most $n-k+1$, suggesting the dimension in an analogue of Theorem~\ref{thm:liang}. However, turning this calculation for fixed contact points into a rectifiability theorem requires accounting for the variation of the nearest points with the center. Our first result establishes this rectifiability conclusion for the dual medial axes. We then examine under what additional assumptions they coincide with the ordinary higher-order medial axes.

\begin{theorem}\label{thm:intro-dual}
    Let $X$ be an $n$-dimensional real normed vector space whose norm is $C^1$ on $X\setminus\{0\}$, and let $E\subseteq X$ be nonempty and closed. For every integer $1\le k\le n+1$, the dual $k$-medial axis $M_k^J(E)$ is $(n-k+1)$-rectifiable. In particular, its Hausdorff dimension satisfies
    \[
        \dim M_k^J(E)\le n-k+1.
    \]
\end{theorem}

We give a direct geometric proof of Theorem~\ref{thm:intro-dual}. Working with contact vectors and contact covectors, we show how affine independence of the latter yields a quantitative cone-exclusion condition and hence containment in local Lipschitz graphs. The proof identifies the subspaces over which these graphs are defined and controls the aperture of the excluded cones, making explicit the role of the contact covectors in the local geometry underlying rectifiability.

The rectifiability conclusion can also be deduced from the theorem of Alberti, Ambrosio, and Cannarsa~\cite{Alberti1992} on the rectifiability of subdifferential strata of semiconvex functions. Indeed, $d_E^2/2$ is locally semiconcave on $X$ with a general modulus, so the superdifferential characterization of $M_k^J(E)$ described above allows us to apply their theorem after changing sign. We give the details in Section~\ref{sec:semiconcavity}.

We next compare the ordinary and dual medial axes. Under the additional assumption of strict convexity, the two notions coincide for $k=2,3$. Consequently, the conclusion of Theorem~\ref{thm:liang} remains valid for the ordinary $2$- and $3$-medial axes in this class of norms.

\begin{theorem}\label{thm:intro-low-rank}
    Let $n\ge 2$, and let $X$ be an $n$-dimensional real normed vector space whose norm is strictly convex and $C^1$ on $X\setminus\{0\}$. Let $E\subseteq X$ be nonempty and closed. Then
    \[
        M_2(E)=M_2^J(E)
        \qquad\text{and}\qquad
        M_3(E)=M_3^J(E).
    \]
    Hence $M_2(E)$ is $(n-1)$-rectifiable and $M_3(E)$ is $(n-2)$-rectifiable. In particular, if $n=2$, then $M_2(E)$ is $1$-rectifiable and $M_3(E)$ is countable.
\end{theorem}

For four contact vectors, however, affine independence need no longer be preserved by $J$, and the rectifiability conclusion for ordinary medial axes can fail. The infinity norm example above exhibits failure for a nonsmooth norm with flat faces, but neither feature is necessary: the following result shows that failure can occur even for smooth strictly convex norms with positively curved unit spheres arbitrarily $C^\infty$-close to the Euclidean sphere $\Sp^2$ in $\R^3$. We adopt the convention that the Euclidean unit sphere has principal curvatures $+1$.

\begin{theorem}\label{thm:intro-smoothexample}
    There exist $\eta_0>0$ and a family of norms $(\norm{\cdot}_\eta)_{|\eta|<\eta_0}$ on $\R^3$, with unit balls denoted by $B_\eta$, such that the following properties hold:
    \begin{enumerate}[label=\textup{(\roman*)}]
        \item $\norm{\cdot}_0$ is the Euclidean norm;
        \item $\norm{\cdot}_\eta$ is strictly convex and $C^\infty$ on $\R^3\setminus\{0\}$, and $\partial B_\eta$ has positive Euclidean principal curvatures;
        \item if $\rho_\eta:\Sp^2\to(0,\infty)$ denotes the radial function of $B_\eta$, then $\rho_\eta\to1$ in $C^\infty(\Sp^2)$ as $\eta\to0$;
        \item if $\eta\ne0$, then there exists a four-point set $E_\eta\subseteq\R^3$ such that the $4$-medial axis $M_4(E_\eta)$, computed with respect to $\norm{\cdot}_\eta$, contains a nondegenerate line segment; hence $M_4(E_\eta)$ is not $0$-rectifiable.
    \end{enumerate}
\end{theorem}

The construction uses the homothetic-tetrahedron mechanism described by Icking, Klein, L\^e, and Ma~\cite[Section~5]{Icking1995}. Here we realize this mechanism on smooth positively curved unit spheres arbitrarily $C^\infty$-close to the Euclidean sphere and compute the resulting loss of affine dimension of the contact covectors. Theorem~\ref{thm:intro-smoothexample} implies that the direct analogue of Theorem~\ref{thm:liang} for ordinary medial axes fails already for $(n,k)=(3,4)$, even under these regularity and curvature assumptions. Along the line segment in our construction, the four contact vectors are affinely independent, whereas their contact covectors have affine dimension exactly $2$. We verify this in Section~\ref{sec:smoothexamples}. The construction therefore exhibits the loss of affine dimension that separates ordinary and dual higher-order medial axes.

The rest of the article is organized as follows. Section~\ref{sec:preliminaries} introduces the notation and preliminary results. Section~\ref{sec:dual-theorem} gives the direct geometric proof of Theorem~\ref{thm:intro-dual}. Section~\ref{sec:semiconcavity} establishes the superdifferential characterization of the dual medial axes and derives an alternative proof from the theorem of Alberti, Ambrosio, and Cannarsa. Section~\ref{sec:ordinary} proves Theorem~\ref{thm:intro-low-rank}. Section~\ref{sec:smoothexamples} constructs the smooth three-dimensional example of Theorem~\ref{thm:intro-smoothexample} in which the affine dimension decreases under the duality map, yielding a failure of rectifiability for the ordinary $4$-medial axis, and explains its extension to higher dimensions and higher orders. Finally, Section~\ref{sec:questions} discusses several open problems.

\paragraph{Acknowledgements.}
This work is supported by the National Natural Science Foundation of China (Grant No.~12671108 and Grant No.~12271018) and Beijing Natural Science Foundation Youth Science Fund Key Project (Grant No.~JQ26002). The first author thanks Beihang University for its hospitality during his visit, when part of this work was carried out.

\section{Preliminaries}\label{sec:preliminaries}

Throughout the following sections, $X$ is an $n$-dimensional real normed space ($n\ge1$) whose norm $\norm{\cdot}$ is $C^1$ on $X\setminus\{0\}$. We retain the notation of Section~\ref{sec:intro}: the dual space and norm $X^*$ and $\dualnorm{\cdot}$, the unit spheres $\Sp_X$ and $\Sp_{X^*}$, the duality map $J$, and the distance function $d_E$ for nonempty closed $E\subseteq X$.

\subsection{Auxiliary Euclidean norm, dimensions, and rectifiability}

We fix an auxiliary Euclidean norm $\enorm{\cdot}$ on $X$, and we write $\enorm{\cdot}_*$ for the corresponding dual norm on $X^*$. By equivalence of norms in finite dimensions, there exists $C_X\ge1$ such that
\begin{equation}\label{Lip_constant}
    C_X^{-1}\enorm{\cdot}\le\norm{\cdot}\le C_X\enorm{\cdot}.
\end{equation}
We write $\mathcal H^m$ for the $m$-dimensional Hausdorff measure induced by $\enorm{\cdot}$ and retain $\dim$ for Hausdorff dimension, with the convention $\dim\varnothing=0$. We use the following convention for rectifiability.

\begin{definition}
    Let $m$ be an integer with $0\le m\le n$. A set $A\subseteq X$ is called \emph{$m$-rectifiable} if there exist Lipschitz maps $f_i:\R^m\to X$, $i\ge1$, such that
    \[
        \mathcal H^m\left(A\setminus\bigcup_{i=1}^{\infty}f_i(\R^m)\right)=0.
    \]
\end{definition}

The empty set is $m$-rectifiable for every $0\le m\le n$. In particular, $0$-rectifiability is equivalent to countability, including the empty set. The equivalence~\eqref{Lip_constant} implies that the two norms induce the same class of Lipschitz maps and the same Hausdorff-measure null sets in each dimension. Consequently, rectifiability is independent of the auxiliary Euclidean norm.

Let $A\subseteq X$ be nonempty and choose any $a\in A$. Its affine hull is $\aff A=a+\spann(A-a)$, and its affine dimension is the algebraic dimension of $\spann(A-a)$. Since Hausdorff dimension agrees with algebraic dimension for linear and affine subspaces, we have $\dim\aff A=\dim\spann(A-a)$. Both the affine hull and the affine dimension are independent of the choice of $a$.

The following elementary criterion allows us to deduce rectifiability from cone avoidance.

\begin{lemma}\label{lem:cone-criterion}
    Let $Q,V\subseteq X$ be complementary linear subspaces, let $\pi_Q:X\to Q$ denote the projection onto $Q$ along $V$, and let $A\subseteq X$. Suppose that, for some $\alpha>0$,
    \[
        (A-A)\cap\{q+v\ne0:q\in Q,\ v\in V,\ \enorm{q}<\alpha\enorm{v}\}=\varnothing.
    \]
    Then $\pi_Q|_A$ is injective, and there exists an $\alpha^{-1}$-Lipschitz map $f:Q\to V$ such that
    \[
        A=\{q+f(q):q\in\pi_Q(A)\}.
    \]
    In particular, $A$ is contained in a Lipschitz graph over $Q$ and is $\dim Q$-rectifiable.
\end{lemma}

\begin{proof}
    Let $\pi_V:X\to V$ denote the projection onto $V$ along $Q$. For $x,y\in A$, the assumption gives
    \[
        \enorm{\pi_Q(y-x)}\ge\alpha\enorm{\pi_V(y-x)}.
    \]
    Thus $\pi_Q(x)=\pi_Q(y)$ implies $x=y$, and the map $f_0=\pi_V\circ(\pi_Q|_A)^{-1}:\pi_Q(A)\to V$ is well defined and $\alpha^{-1}$-Lipschitz. By Kirszbraun's extension theorem~\cite{AzagraLeGruyerMudarra2021}, it extends to an $\alpha^{-1}$-Lipschitz map $f:Q\to V$. For $A=\varnothing$, one may simply take $f=0$. By construction, $A$ is the set of points of the form $q+f(q)$ with $q\in\pi_Q(A)$. The map $q\mapsto q+f(q)$ is Lipschitz on $Q$ and its image thus contains $A$. Identifying $Q$ linearly with $\R^{\dim Q}$ proves the rectifiability assertion.
\end{proof}

\subsection{Basic properties of the duality map}

We next establish basic properties of $J$ used in the proofs. Since the norm is $C^1$ on $X\setminus\{0\}$, its Fréchet derivative $D\norm{\cdot}(u)\in X^*$ is well defined at every $u\ne0$. By definition,
\begin{equation}\label{eq:Frechet}
    \lim_{h\to0}\frac{\norm{u+h}-\norm{u}-D\norm{\cdot}(u)(h)}{\norm{h}}=0.
\end{equation}
We defined $J$ as the restriction of the map $D\norm{\cdot}\colon X\setminus\{0\}\to X^*$ to $\Sp_X$. The $C^1$ regularity of the norm implies that $J$ is continuous on $\Sp_X$. The following lemma justifies viewing $J$ as a map $\Sp_X\to\Sp_{X^*}$.

\begin{lemma}\label{lem:duality-identities}
    For every $u\in\Sp_X$ and $r>0$, it holds that
    \[
        J(u)(u)=1,
        \qquad
        \dualnorm{J(u)}=1,
        \qquad
        J(-u)=-J(u),
        \qquad
        D\norm{\cdot}(ru)=J(u).
    \]
    Moreover, $J(u)$ is the unique norming functional of $u$, that is, the unique $f\in X^*$ such that $\dualnorm{f}=1$ and $f(u)=1$.
\end{lemma}

\begin{proof}
    Evaluating~\eqref{eq:Frechet} along lines gives, for every $u\ne0$ and $h\in X$,
    \[
        D\norm{\cdot}(u)(h)=\lim_{t\to0}\frac{\norm{u+th}-\norm{u}}{t}.
    \]
    For $u\in\Sp_X$, this yields $J(u)(u)=\norm{u}=1$ and, by the triangle inequality, $|J(u)(h)|\le\norm{h}$ for every $h\in X$. It follows that $\dualnorm{J(u)}=1$. This also directly gives $J(-u)(h)=J(u)(-h)=-J(u)(h)$ and $D\norm{\cdot}(ru)(h)=J(u)(h)$ for every $h\in X$; hence $J(-u)=-J(u)$ and $D\norm{\cdot}(ru)=J(u)$. Finally, if $f\in X^*$ satisfies $\dualnorm{f}=1$ and $f(u)=1$, then for every $h\in X$,
    \[
        1+t f(h)=f(u+th)\le\norm{u+th}=1+tJ(u)(h)+o(|t|)
        \qquad\text{as } t\to0.
    \]
    Letting $t\to0$ through positive and negative values gives $f(h)=J(u)(h)$. Thus $f=J(u)$.
\end{proof}

\subsection{Nearest-point configurations and medial axes}

We now introduce notation for the nearest-point configurations described in the introduction. Let $E\subseteq X$ be nonempty and closed. For $x\in X$, write
\[
    \Phi_E(x)=\{y\in E:\norm{x-y}=d_E(x)\}.
\]
This set is nonempty and compact. Indeed, every sequence in $E$ whose distances to $x$ converge to $d_E(x)$ is bounded and therefore has a convergent subsequence. Its limit belongs to $E$ and realizes the distance. Moreover, $\Phi_E(x)$ is closed and bounded, hence compact.

For $x\notin E$, collect the contact vectors and contact covectors into the sets
\[
    U_E(x)=\left\{\frac{y-x}{d_E(x)}:y\in\Phi_E(x)\right\}\subseteq\Sp_X
    \qquad\text{and}\qquad
    N_E(x)=J(U_E(x))\subseteq\Sp_{X^*}.
\]
Both sets are nonempty and compact, by compactness of $\Phi_E(x)$ and continuity of $J$. Moreover, since $y\mapsto(y-x)/d_E(x)$ is an invertible affine map on $X$, we have
\[
    \dim\aff\Phi_E(x)=\dim\aff U_E(x).
\]

The definitions from the introduction can now be expressed in terms of these contact sets.

\begin{definition}\label{def:medial-axes}
    For $1\le k\le n+1$, the \emph{$k$-medial axis} and the \emph{dual $k$-medial axis} of $E$ are given by
    \[
        M_k(E)=\{x\notin E:\dim\aff U_E(x)\ge k-1\}
        \quad\text{and}\quad
        M_k^J(E)=\{x\notin E:\dim\aff N_E(x)\ge k-1\},
    \]
    respectively.
\end{definition}

Finally, to make the Euclidean identification explicit, suppose that the norm is induced by an inner product and let $R:X\to X^*$ be the associated Riesz isomorphism, $R(u)=\langle u,\cdot\rangle$. Then $J=R|_{\Sp_X}$ and $N_E(x)=R(U_E(x))$. Since $R$ preserves affine dimension, $M_k^J(E)=M_k(E)$ for every $1\le k\le n+1$.

\subsection{The squared norm and uniform first-order estimates}

It will also be convenient to work with the squared norm. For every $x\in X$, define
\begin{equation}\label{eq:squared-norm}
    F(x)=\frac{\norm{x}^2}{2}.
\end{equation}
Then $F\in C^1(X)$. Indeed, $F(h)=o(\norm{h})$ as $h\to0$, so $DF(0)=0$. Moreover, the chain rule gives $DF(x)=\norm{x}D\norm{\cdot}(x)$ for $x\ne0$. Hence, by Lemma~\ref{lem:duality-identities}, for $u\in\Sp_X$ and $r>0$,
\begin{equation}\label{eq:squared-norm-property}
    DF(ru)=rJ(u)
    \qquad\text{and}\qquad
    DF(-ru)=-rJ(u).
\end{equation}
In particular, $\dualnorm{DF(x)}=\norm{x}$ for $x\ne0$, which proves continuity of $DF$ at the origin.

We shall repeatedly use uniform first-order estimates of the squared norm $F$ on compact sets.

\begin{lemma}\label{lem:uniform-expansion}
    Let $K\subseteq X$ be nonempty and compact. There exists a continuous nondecreasing function $\omega_K:[0,\infty)\to[0,\infty)$ with $\omega_K(0)=0$ such that, for every $z\in K$ and $h\in X$,
    \[
        \left|F(z+h)-F(z)-DF(z)(h)\right|\le \omega_K(\enorm{h})\enorm{h}.
    \]
\end{lemma}

\begin{proof}
    For every $t\ge0$, define
    \[
        \omega_K(t)=\sup_{(z,z')\in A_t}\enorm{DF(z')-DF(z)}_*
        \qquad\text{where}\qquad
        A_t=\left\{(z,z')\in K\times X:\enorm{z'-z}\le t\right\}.
    \]
    The set $A_t$ is compact. Indeed, it is closed, and $z'$ remains in a bounded neighborhood of the compact set $K$. Since $DF:X\to X^*$ is continuous, the supremum defining $\omega_K(t)$ is therefore finite and attained. Clearly, $\omega_K$ is nonnegative and nondecreasing, and $\omega_K(0)=0$.
    
    We next verify that $\omega_K$ is continuous. Let $t_j\to t\ge0$. For each $j$, choose $(z_j,z_j')\in A_{t_j}$ such that $\omega_K(t_j)=\enorm{DF(z_j')-DF(z_j)}_*$. After passing to a subsequence realizing the limsup, the pairs $(z_j,z_j')$ lie in the compact set $A_{t+1}$ for all sufficiently large $j$. Passing to a further subsequence, we may assume that $z_j\to z\in K$ and $z_j'\to z'$. Then $\enorm{z'-z}\le t$, and hence, by continuity of $DF$,
    \[
        \limsup_{j\to\infty}\omega_K(t_j)=\enorm{DF(z')-DF(z)}_*\le\omega_K(t).
    \]
    Conversely, choose $(z,z')\in A_t$ attaining $\omega_K(t)$. If $z=z'$, then $\omega_K(t)=0$ and the reverse inequality is immediate. Otherwise, set $\lambda_j=\min\{1,t_j/\enorm{z'-z}\}$ and $z_j'=z+\lambda_j(z'-z)$. Then $(z,z_j')\in A_{t_j}$ and $z_j'\to z'$, so
    \[
        \liminf_{j\to\infty}\omega_K(t_j)\ge\lim_{j\to\infty}\enorm{DF(z_j')-DF(z)}_*=\omega_K(t).
    \]
    Thus $\omega_K(t_j)\to\omega_K(t)$, proving the continuity of $\omega_K$.

    Finally, let $z\in K$ and $h\in X$. By the fundamental theorem of calculus,
    \[
        F(z+h)-F(z)-DF(z)(h)=\int_0^1\bigl(DF(z+sh)-DF(z)\bigr)(h)\,ds.
    \]
    For every $s\in[0,1]$, we have $\enorm{z+sh-z}\le\enorm{h}$, so the definition of $\omega_K$ gives
    \[
        \left|F(z+h)-F(z)-DF(z)(h)\right|\le\int_0^1\enorm{DF(z+sh)-DF(z)}_*\,\enorm{h}\,ds\le\omega_K(\enorm{h})\enorm{h}.
    \]
    This proves the lemma.
\end{proof}

\section{Rectifiability of the dual medial axes}\label{sec:dual-theorem}

In this section, we prove Theorem~\ref{thm:intro-dual}. Fix a nonempty closed set $E\subseteq X$. If $E=X$, then $M_k^J(E)=\varnothing$ for every $k$, so the conclusion is immediate. We may therefore assume that $E\ne X$. Moreover, the case $k=1$ is immediate: since $N_E(x)$ is nonempty for every $x\notin E$, we have $M_1^J(E)=X\setminus E$. Every subset of the $n$-dimensional space $X$ is $n$-rectifiable, and hence $\dim M_1^J(E)\le n$. We therefore fix an integer $2\le k\le n+1$ for the remainder of the proof.

\subsection{Quantitative separation of covector configurations}

We consider reference configurations of unit vectors whose images under the duality map are affinely independent. For each configuration, we identify a cone of directions on which these covectors have a uniformly positive spread, stable under small perturbations of the unit vectors. In the next subsection, we apply this estimate to actual contact configurations to obtain cone exclusion.

Let $\mathcal A_k^J$ denote the set of ordered $k$-tuples of unit vectors whose images under the duality map are affinely independent:
\[
    \mathcal A_k^J=\left\{a=(a_1,\ldots,a_k)\in\Sp_X^k:
    J(a_1),\ldots,J(a_k)\text{ are affinely independent}\right\}.
\]
Fix $a\in\mathcal A_k^J$. Since $J(a_1),\ldots,J(a_k)$ are affinely independent, the functionals $J(a_i)-J(a_1)$, $2\le i\le k$, are linearly independent. Letting $\ker$ denote the kernel of a linear functional, define
\begin{equation}\label{eq:Qa}
    Q_a=\bigcap_{i=2}^k\ker\bigl(J(a_i)-J(a_1)\bigr).
\end{equation}
Then $Q_a$ has dimension $n-k+1$. Choose a linear complement $V_a$ of $Q_a$, so that $X=Q_a\oplus V_a$ and $\dim V_a=k-1$. For $h\in X$, define the spread of the covectors $J(a_1),\ldots,J(a_k)$ in the direction $h$ by
\[
    \sigma_a(h)=\max_{1\le i\le k}J(a_i)(h)-\min_{1\le i\le k}J(a_i)(h).
\]
The function $\sigma_a$ is continuous, even, and positively homogeneous. Moreover,
\[
    \sigma_a(h)=0
    \quad\Longleftrightarrow\quad
    J(a_i)(h)=J(a_1)(h)\ \text{for every }i
    \quad\Longleftrightarrow\quad
    h\in Q_a.
\]
Thus the spread $\sigma_a$ vanishes precisely on $Q_a$ and is positive on $V_a\setminus\{0\}$. In addition, define
\[
    \mathcal C_a=\{q+v\ne0:q\in Q_a,\ v\in V_a,\ \enorm{q}<\enorm{v}\}.
\]
Then $\mathcal C_a$ is a symmetric open cone around $V_a\setminus\{0\}$. The following lemma gives a uniform lower bound for $\sigma_a$ on $\mathcal C_a$ and shows that the bound persists when the vectors $a_i$ are perturbed.

\begin{lemma}\label{lem:spread}
    Let $a\in\mathcal A_k^J$. There exist constants $\gamma_a,\varepsilon_a>0$ such that, for all $u=(u_1,\ldots,u_k)$ and $w=(w_1,\ldots,w_k)$ in $\Sp_X^k$, and for all $h\in\mathcal C_a$,
    \[
        \max_{1\le i\le k}\max\{\enorm{u_i-a_i},\enorm{w_i-a_i}\}<\varepsilon_a
        \quad\Longrightarrow\quad
        \max_{1\le i\le k}J(u_i)(h)-\min_{1\le i\le k}J(w_i)(h)\ge\gamma_a\enorm{h}.
    \]
\end{lemma}

\begin{proof}
    Let $S_a=\{v\in V_a:\enorm{v}=1\}$. Since $\sigma_a$ is continuous and positive on the compact set $S_a$,
    \[
        c_a=\min_{v\in S_a}\sigma_a(v)>0.
    \]
    For $q\in Q_a$, the values $J(a_i)(q)$ are independent of $i$, so $\sigma_a(q+v)=\sigma_a(v)$ for every $v\in V_a$. Thus, if $h=q+v\in\mathcal C_a$, then $v\ne0$, and positive homogeneity of $\sigma_a$ together with the triangle inequality gives
    \[
        \sigma_a(h)=\sigma_a(v)=\enorm{v}\sigma_a\left(\frac{v}{\enorm{v}}\right)\ge c_a\enorm{v}\ge\frac{c_a}{2}\left(\enorm{q}+\enorm{v}\right)\ge\frac{c_a}{2}\enorm{h}.
    \]
    
    Since $\Sp_X$ is compact and $J$ is continuous, $J$ is uniformly continuous on $\Sp_X$. We may therefore choose $\varepsilon_a>0$ such that, for all $x,y\in\Sp_X$,
    \[
        \enorm{x-y}<\varepsilon_a
        \quad\Longrightarrow\quad
        \enorm{J(x)-J(y)}_*\le\frac{c_a}{8}.
    \]

    Now let $u=(u_1,\ldots,u_k)$ and $w=(w_1,\ldots,w_k)$ be in $\Sp_X^k$ such that $\enorm{u_i-a_i}<\varepsilon_a$ and $\enorm{w_i-a_i}<\varepsilon_a$ for every $i$. Then, for every $h\in\mathcal C_a$,
    \[
        \max_{1\le i\le k}J(u_i)(h)\ge\max_{1\le i\le k}J(a_i)(h)-\frac{c_a}{8}\enorm{h}
        \quad\text{and}\quad
        \min_{1\le i\le k}J(w_i)(h)\le\min_{1\le i\le k}J(a_i)(h)+\frac{c_a}{8}\enorm{h}.
    \]
    Hence
    \[
        \max_{1\le i\le k}J(u_i)(h)-\min_{1\le i\le k}J(w_i)(h)
        \ge \sigma_a(h)-\frac{c_a}{4}\enorm{h}
        \ge \frac{c_a}{4}\enorm{h}.
    \]
    Thus the result holds with $\gamma_a=c_a/4$.
\end{proof}

\begin{remark}
    The same choice of $\varepsilon_a$ ensures that $J(u_1),\ldots,J(u_k)$ are affinely independent whenever $\enorm{u_i-a_i}<\varepsilon_a$ for every $i$. Indeed, otherwise the subspace $Q_u$ defined by~\eqref{eq:Qa} with $u_i$ in place of $a_i$ would have dimension at least $n-k+2$. Since $\dim V_a=k-1$, we have $\dim(Q_u\cap V_a)\ge\dim Q_u+\dim V_a-n\ge1$. Choose $h\in(Q_u\cap V_a)\setminus\{0\}$. Then $h\in\mathcal C_a$ and $J(u_1)(h)=\cdots=J(u_k)(h)$, contradicting Lemma~\ref{lem:spread} applied with $w=u$.
\end{remark}

\subsection{A cone-exclusion argument}

We now localize the dual medial axis according to the distance to $E$ and the configuration of contact vectors. For a fixed $a\in\mathcal A_k^J$, we will show that points with contact vectors close to $a$ cannot approach one another in directions belonging to the cone $\mathcal C_a$. This cone-exclusion property will yield the local Lipschitz graph structure in the next subsection.

Let $a=(a_1,\ldots,a_k)\in\mathcal A_k^J$ and $d>0$, and let $\varepsilon_a>0$ be the constant given by Lemma~\ref{lem:spread}. Define
\begin{align*}
    \cM(a,d)
    =\bigl\{x\notin E:\,
    &|d_E(x)-d|<d/2,\ \text{and for each }1\le i\le k,\\
    &\text{there exists }u_i\in U_E(x)\text{ such that }\enorm{u_i-a_i}<\varepsilon_a
    \bigr\}.
\end{align*}
The remark following the proof of Lemma~\ref{lem:spread} shows that the corresponding covectors $J(u_1),\ldots,J(u_k)$ are affinely independent; hence $\cM(a,d)\subseteq M_k^J(E)$.

\begin{lemma}\label{lem:cone-exclusion}
    Let $a\in\mathcal A_k^J$ and $d>0$. There exists $\tau_{a,d}>0$ such that for all $x,y\in\cM(a,d)$,
    \[
        \enorm{y-x}\le\tau_{a,d}\quad\Longrightarrow\quad y-x\notin\mathcal C_a.
    \]
\end{lemma}

\begin{proof}
    Let $\gamma_a>0$ be the constant given by Lemma~\ref{lem:spread}. Let $K$ be the compact annulus consisting of all points $\xi\in X$ such that $d/2\le\norm{\xi}\le3d/2$, and let $\omega_K$ be the modulus given by Lemma~\ref{lem:uniform-expansion}. Since $\omega_K(\tau)\to0$ as $\tau\to0$, we may choose $\tau_{a,d}>0$ small enough to ensure that for every $0\le\tau\le\tau_{a,d}$,
    \[
        C_X^2\tau+2\omega_K(\tau)\le\frac{d\gamma_a}{2}.
    \]

    Suppose, toward a contradiction, that there exist $x,y\in\cM(a,d)$ with $h=y-x\in\mathcal C_a$ and $\enorm{h}\le\tau_{a,d}$. Choose $u_1,\ldots,u_k\in U_E(x)$ and $w_1,\ldots,w_k\in U_E(y)$ such that $\enorm{u_i-a_i}<\varepsilon_a$ and
    $\enorm{w_i-a_i}<\varepsilon_a$ for every $i$. Since $h\in\mathcal C_a$, Lemma~\ref{lem:spread} yields indices $i,j$ such that
    \[
        J(u_i)(h)-J(w_j)(h)\ge\gamma_a\enorm{h}.
    \]
    The corresponding nearest points are $p=x+ru_i\in E$ and $q=y+sw_j\in E$, where $r=d_E(x)$ and $s=d_E(y)$. Note that $|s-r|\le\norm{h}$ because $d_E$ is $1$-Lipschitz with respect to $\norm{\cdot}$. Since $\dualnorm{J(w_j)}=1$, we deduce that
    \begin{align*}
        -rJ(u_i)(h)+sJ(w_j)(h)&=-r\left(J(u_i)(h)-J(w_j)(h)\right)+(s-r)J(w_j)(h)\\
        &\le-r\gamma_a\enorm{h}+\norm{h}^2\le-r\gamma_a\enorm{h}+C_X^2\enorm{h}^2.
    \end{align*}

    Since $|r-d|,|s-d|<d/2$, both $ru_i$ and $sw_j$ belong to $K$. Applying Lemma~\ref{lem:uniform-expansion} at these points and using~\eqref{eq:squared-norm-property} gives
    \[
        F(ru_i-h)=\frac{r^2}{2}-rJ(u_i)(h)+R_1
        \quad\text{and}\quad
        F(sw_j+h)=\frac{s^2}{2}+sJ(w_j)(h)+R_2,
    \]
    where $|R_1|+|R_2|\le2\omega_K(\enorm{h})\enorm{h}$. In addition, since $p,q\in E$, we have
    \[
        r=d_E(x)\le\norm{x-q}=\norm{sw_j+h}
        \qquad\text{and}\qquad
        s=d_E(y)\le\norm{y-p}=\norm{ru_i-h}.
    \]
    It follows that
    \begin{align*}
        0&\le F(sw_j+h)-\frac{r^2}{2}=\frac{s^2}{2}-\frac{r^2}{2}+sJ(w_j)(h)+R_2\le F(ru_i-h)-\frac{r^2}{2}+sJ(w_j)(h)+R_2\\
        &=-rJ(u_i)(h)+R_1+sJ(w_j)(h)+R_2\le-r\gamma_a\enorm{h}+C_X^2\enorm{h}^2+2\omega_K(\enorm{h})\enorm{h}\le\left(\frac{d}{2}-r\right)\gamma_a\enorm{h}.
    \end{align*}
    Since $r>d/2$ and $\gamma_a\enorm{h}>0$, this is a contradiction. The lemma follows.
\end{proof}

\subsection{Completion of the proof of Theorem~\ref{thm:intro-dual}}

We can now complete the proof of Theorem~\ref{thm:intro-dual}, assuming that $E\ne X$ and $k\ge2$, as mentioned at the beginning of the section. Fix $a\in\mathcal A_k^J$ and $d>0$. By Lemma~\ref{lem:cone-exclusion}, there exists $\tau_{a,d}>0$ such that any two points of $\cM(a,d)$ at Euclidean distance at most $\tau_{a,d}$ have difference outside $\mathcal C_a$. Cover $X$ by a countable family of sets of Euclidean diameter at most $\tau_{a,d}$. For every such set $B$, the set $A_B=\cM(a,d)\cap B$ satisfies $(A_B-A_B)\cap\mathcal C_a=\varnothing$. By the form of $\mathcal C_a$ and Lemma~\ref{lem:cone-criterion}, $A_B$ is $(n-k+1)$-rectifiable. Taking the countable union over all such sets $B$ proves that $\cM(a,d)$ is $(n-k+1)$-rectifiable.

It remains to cover $M_k^J(E)$ by countably many such quantitative pieces. The parameter space $\mathcal A_k^J\times(0,\infty)$ is second countable, hence Lindelöf. For $(a,d)\in\mathcal A_k^J\times(0,\infty)$, let
\[
    \mathcal B(a,d)=
    \left\{(u,r)\in\mathcal A_k^J\times(0,\infty):
    \enorm{u_i-a_i}<\varepsilon_a\ \text{for every }i\text{, and }|r-d|<d/2\right\}.
\]
These sets form an open cover of $\mathcal A_k^J\times(0,\infty)$. Hence there is a countable subcover $\{\mathcal B(a^\ell,d^\ell)\}_{\ell\ge1}$, where $a^\ell=(a_1^\ell,\ldots,a_k^\ell)$. Now let $x\in M_k^J(E)$. By definition, there exist $u_1,\ldots,u_k\in U_E(x)$ such that $J(u_1),\ldots,J(u_k)$ are affinely independent. Thus $u=(u_1,\ldots,u_k)\in\mathcal A_k^J$, and $(u,d_E(x))\in\mathcal A_k^J\times(0,\infty)$. Hence $(u,d_E(x))\in\mathcal B(a^\ell,d^\ell)$ for some $\ell$, which implies $x\in\cM(a^\ell,d^\ell)$. Since each $\cM(a^\ell,d^\ell)$ is contained in $M_k^J(E)$, we obtain
\[
    M_k^J(E)=\bigcup_{\ell=1}^{\infty}\cM(a^\ell,d^\ell).
\]
Each set on the right is $(n-k+1)$-rectifiable, and thus so is $M_k^J(E)$. In particular, $\dim M_k^J(E)\le n-k+1$.

\section{Dual medial axes and semiconcavity}\label{sec:semiconcavity}

We now relate the dual medial axes to the singularities of the squared distance function. We identify $M_k^J(E)$ with the set where the superdifferential of $d_E^2/2$ has dimension at least $k-1$, and deduce an alternative proof of Theorem~\ref{thm:intro-dual} from the theorem of Alberti, Ambrosio, and Cannarsa~\cite{Alberti1992}. The results of~\cite{Alberti1992} are formulated in $\R^n$. We state the relevant definitions directly on $X$ and explain below how to apply their theorem through a linear isometry from $(X,\enorm{\cdot})$ onto Euclidean $\R^n$.

We first recall the notions of superdifferential and subdifferential.

\begin{definition}
    Let $S\subseteq X$ be open and let $u:S\to\R$. For $x\in S$, define
    \begin{align*}
        &D^+u(x)=\left\{p\in X^*:\limsup_{y\to x}\frac{u(y)-u(x)-p(y-x)}{\norm{y-x}}\le0\right\}\\
        \text{and}\quad&D^-u(x)=\left\{p\in X^*:\liminf_{y\to x}\frac{u(y)-u(x)-p(y-x)}{\norm{y-x}}\ge0\right\}.
    \end{align*}
    These sets are called, respectively, the \emph{superdifferential} and the \emph{subdifferential} of $u$ at $x$.
\end{definition}

We also recall the notion of semiconcavity with a general modulus.

\begin{definition}
    Let $S\subseteq X$. A function $u:S\to\R$ is \emph{semiconcave} if there exists a nondecreasing upper semicontinuous function $\omega:[0,\infty)\to[0,\infty)$, with $\omega(\rho)\to0$ as $\rho\to0$, such that
    \[
        \lambda u(x)+(1-\lambda)u(y)-u\bigl(\lambda x+(1-\lambda)y\bigr)
        \le\lambda(1-\lambda)\norm{x-y}\omega(\norm{x-y})
    \]
    whenever $x,y\in S$, $[x,y]\subseteq S$, and $\lambda\in[0,1]$. The function $\omega$ is called a \emph{modulus of semiconcavity}. A function $u$ is \emph{semiconvex} if $-u$ is semiconcave, and is \emph{locally semiconcave} on a set $S$ if it is semiconcave on every compact subset of $S$.
\end{definition}

We will use the following result of Alberti, Ambrosio, and Cannarsa.

\begin{theorem}[{Alberti--Ambrosio--Cannarsa~\cite[Theorem~4.1]{Alberti1992}}]\label{thm:AAC}
    Let $\Omega\subseteq\R^n$ be bounded, convex, and open, and let $u:\Omega\to\R$ be semiconvex and Lipschitz continuous. For every integer $0\le m\le n$, the set
    \[
        \Sigma^m(u)=\{x\in\Omega:\dim D^-u(x)\ge m\}
    \]
    is $(n-m)$-rectifiable.
\end{theorem}

Although the preceding notions are defined using $\norm{\cdot}$, replacing it by the equivalent Euclidean norm $\enorm{\cdot}$ leaves the subdifferentials and superdifferentials unchanged and preserves semiconcavity and semiconvexity after adjusting the modulus. Lipschitz continuity is also preserved. Moreover, a linear identification of $X$ with $\R^n$ preserves convexity, openness, and boundedness. We may therefore apply Theorem~\ref{thm:AAC} to bounded convex open subsets of $X$.

Recall that $F$ is the squared norm defined in~\eqref{eq:squared-norm}. For $x\in X$, define
\[
    g_E(x)=\frac{d_E(x)^2}{2}=\min_{y\in E}F(x-y).
\]
We will relate its superdifferential outside $E$ to the contact covectors.

\begin{lemma}\label{lem:gEconcave}
    The function $g_E$ is locally semiconcave on $X$.
\end{lemma}

\begin{proof}
    By the norm-invariance just discussed, it suffices to prove semiconcavity with respect to the auxiliary Euclidean norm. Let $K$ be a nonempty compact subset of $X$, and set
    \[
        d_1=\sup_{x_1,x_2\in K}\enorm{x_1-x_2}
        \qquad\text{and}\qquad
        d_2=\sup_{x\in K}d_E(x).
    \]
    Consider the compact Euclidean ball $B=\{w\in X:\enorm{w}\le d_1+C_Xd_2\}$. If $x_1,x_2\in K$ and $y_2\in\Phi_E(x_2)$, then $x_1-y_2\in B$. In particular, all the vectors arising below lie in $B$. We first observe that $F$ is semiconcave on $B$. Let $a,b\in B$, $\lambda\in[0,1]$, and $z=\lambda a+(1-\lambda)b$. By Lemma~\ref{lem:uniform-expansion},
    \[
        F(a)-F(z)-DF(z)(a-z)\le\omega_B(\enorm{a-z})\enorm{a-z},
    \]
    and the analogous inequality holds with $b$ in place of $a$. Multiplying these inequalities by $\lambda$ and $1-\lambda$, respectively, and using $\lambda(a-z)+(1-\lambda)(b-z)=0$, we obtain
    \begin{equation}\label{eq:F_semiconcave}
        \lambda F(a)+(1-\lambda)F(b)-F(z)\le
        2\lambda(1-\lambda)\omega_B(\enorm{a-b})\enorm{a-b}.
    \end{equation}
    Thus $F$ is semiconcave on $B$, with modulus $2\omega_B$.

    Now let $x_1,x_2\in K$ with $[x_1,x_2]\subseteq K$, let $z=\lambda x_1+(1-\lambda)x_2$, and choose $y_z\in\Phi_E(z)$. We have $g_E(z)=F(z-y_z)$ and $g_E(x_i)\le F(x_i-y_z)$ for $i=1,2$. The vectors $x_1-y_z$, $x_2-y_z$, and $z-y_z$ belong to $B$, so~\eqref{eq:F_semiconcave} yields
    \[
        \lambda g_E(x_1)+(1-\lambda)g_E(x_2)-g_E(z)
        \le 2\lambda(1-\lambda)\omega_B(\enorm{x_1-x_2})\enorm{x_1-x_2}.
    \]
    Hence $g_E$ is semiconcave on $K$. Since $K$ was arbitrary, $g_E$ is locally semiconcave on $X$.
\end{proof}

The next proposition identifies the superdifferential of $g_E$ in terms of the contact covectors.

\begin{proposition}\label{prop:gEconv}
    For every $x\in X\setminus E$,
    \[
        D^+g_E(x)=\conv\{-d_E(x)J(u):u\in U_E(x)\}.
    \]
    Consequently, for every $1\le k\le n+1$,
    \[
        M_k^J(E)=\{x\in X\setminus E:\dim D^+g_E(x)\ge k-1\}.
    \]
\end{proposition}

\begin{proof}
    Let $x\in X\setminus E$. It suffices to prove that
    \begin{equation}\label{eq:D+gE=convDF(x-y)}
        D^+g_E(x)=\conv\{DF(x-y):y\in\Phi_E(x)\}.
    \end{equation}
    Indeed, the correspondence $u=(y-x)/d_E(x)$ maps $\Phi_E(x)$ bijectively onto $U_E(x)$, and~\eqref{eq:squared-norm-property} gives $DF(x-y)=-d_E(x)J(u)$. This yields the first assertion. Since nonzero scaling and passage to the convex hull preserve affine dimension, and Hausdorff dimension agrees with affine dimension for nonempty convex sets, it also gives
    \[
        \dim D^+g_E(x)=\dim\aff J(U_E(x))=\dim\aff N_E(x).
    \]
    The second assertion then follows from the definition of $M_k^J(E)$.

    We now prove~\eqref{eq:D+gE=convDF(x-y)}. Since $K_x=\{x-y:y\in\Phi_E(x)\}$ is compact, Lemma~\ref{lem:uniform-expansion} gives, for every sufficiently small $h\in X$ and every $y\in\Phi_E(x)$,
    \[
        g_E(x+h)-g_E(x)\le F(x-y+h)-F(x-y)\le DF(x-y)(h)+\omega_{K_x}(\enorm{h})\enorm{h}.
    \]
    Equivalence of the two norms implies that $DF(x-y)\in D^+g_E(x)$. Since $D^+g_E(x)$ is convex, as follows directly from its defining inequality, we obtain the inclusion $\supseteq$ in~\eqref{eq:D+gE=convDF(x-y)}. Now fix $h\in X$. Replacing the displacement by $th$, for $t>0$ small enough, and dividing by $t$, we get
    \[
    \limsup_{t\downarrow0}\frac{g_E(x+th)-g_E(x)}{t}\le\min_{y\in\Phi_E(x)}DF(x-y)(h).
    \]
    
    Choose a sequence $t_j\downarrow0$ along which the difference quotient converges to its $\liminf$ as $t\downarrow0$. For each $j$, choose $y_j\in\Phi_E(x+t_jh)$. Since $\norm{y_j-x}\le d_E(x+t_jh)+t_j\norm{h}$, the sequence $(y_j)$ is bounded. After passing to a subsequence, we may therefore assume that $y_j\to y_0\in E$. Passing to the limit in $\norm{x+t_jh-y_j}=d_E(x+t_jh)$ gives $y_0\in\Phi_E(x)$. The set $L$ consisting of the vectors $x-y_j$, together with their limit $x-y_0$, is compact. Applying Lemma~\ref{lem:uniform-expansion} on $L$ gives
    \[
        g_E(x+t_jh)-g_E(x)\ge F(x+t_jh-y_j)-F(x-y_j)\ge t_jDF(x-y_j)(h)-\omega_L(\enorm{t_jh})\enorm{t_jh}.
    \]
    By continuity of $DF$, dividing by $t_j$ and letting $j\to\infty$ yields
    \[
        \liminf_{t\downarrow0}\frac{g_E(x+th)-g_E(x)}{t}
        \ge DF(x-y_0)(h)
        \ge\min_{y\in\Phi_E(x)}DF(x-y)(h).
    \]
    We have therefore proved that, for every $h\in X$,
    \[
        g_E'(x;h)=\lim_{t\downarrow0}\frac{g_E(x+th)-g_E(x)}{t}
        =\min_{y\in\Phi_E(x)}DF(x-y)(h)=\min_{q\in C}q(h),
    \]
    where $C$ is the set on the right-hand side of~\eqref{eq:D+gE=convDF(x-y)}. Suppose, toward a contradiction, that $p\in D^+g_E(x)\setminus C$. For every $h\ne0$, the definition of the superdifferential gives
    \[
        \frac{g_E'(x;h)-p(h)}{\norm{h}}
        =\lim_{t\downarrow0}\frac{g_E(x+th)-g_E(x)-p(th)}{\norm{th}}
        \le0.
    \]
    Hence $p(h)\ge\min_{q\in C}q(h)$ for every $h\in X$. Since $\Phi_E(x)$ is compact, $DF$ is continuous, and $X^*$ is finite-dimensional, the set $C$ is compact and convex. Strict separation yields $\Lambda\in X^{**}$ such that $\Lambda(p)<\min_{q\in C}\Lambda(q)$. By finite dimensionality, $\Lambda(q)=q(h)$ for some $h\in X\setminus\{0\}$ and every $q\in X^*$. Hence $p(h)<\min_{q\in C}q(h)$, contradicting the inequality established above. This proves $\subseteq$ in~\eqref{eq:D+gE=convDF(x-y)}.
\end{proof}

If $E=X$, all the medial axes are empty and the rectifiability conclusion is immediate. Otherwise, choose a countable cover $(U_\ell)_{\ell\ge1}$ of $X\setminus E$ by open balls whose closures are contained in $X\setminus E$. We now apply Theorem~\ref{thm:AAC}. On each $U_\ell$, the function $-g_E$ is semiconvex by applying Lemma~\ref{lem:gEconcave} to $\overline{U_\ell}$, and is Lipschitz because $d_E$ is Lipschitz and bounded there. Since $D^-(-g_E)=-D^+g_E$, Proposition~\ref{prop:gEconv} and Theorem~\ref{thm:AAC}, applied with $m=k-1$, show that $M_k^J(E)\cap U_\ell$ is $(n-k+1)$-rectifiable for every $1\le k\le n+1$. Taking the countable union over $\ell$ proves the rectifiability conclusion of Theorem~\ref{thm:intro-dual}.

\section{Equivalence of the ordinary and dual medial axes for \texorpdfstring{$k=2,3$}{k=2,3}}\label{sec:ordinary}

In this section, we prove Theorem~\ref{thm:intro-low-rank}. Theorem~\ref{thm:intro-dual} yields a rectifiability result for the ordinary $k$-medial axis whenever affine independence of contact vectors implies affine independence of the corresponding contact covectors. We formalize this property as follows.

\begin{definition}[Rank-preservation property]
    For $2\le k\le n+1$, a norm has \emph{property $\mathrm{RP}_k$} if, for every $a_1,\ldots,a_k\in\Sp_X$,
    \[
        \dim\aff\{a_1,\ldots,a_k\}=k-1
        \quad\Longrightarrow\quad
        \dim\aff\{J(a_1),\ldots,J(a_k)\}=k-1.
    \]
\end{definition}

\begin{corollary}
    If a norm has property $\mathrm{RP}_k$, then, for every nonempty closed set $E\subseteq X$,
    \[
        M_k(E)\subseteq M_k^J(E).
    \]
    Consequently, $M_k(E)$ is $(n-k+1)$-rectifiable.
\end{corollary}

\begin{proof}
    If $x\in M_k(E)$, then $U_E(x)$ contains $k$ affinely independent contact vectors. Property $\mathrm{RP}_k$ implies that their images under $J$ are affinely independent, and hence $x\in M_k^J(E)$. The rectifiability conclusion follows from Theorem~\ref{thm:intro-dual}, since every subset of an $(n-k+1)$-rectifiable set is $(n-k+1)$-rectifiable.
\end{proof}

As observed after Definition~\ref{def:medial-axes} above, when the norm is induced by an inner product, the duality map preserves affine dimension, so $\mathrm{RP}_k$ holds for every $k$ and $M_k^J(E)=M_k(E)$. We now show that, under strict convexity, the same conclusion still holds for the first two nontrivial ranks $k=2$ and $k=3$.

\subsection{The \texorpdfstring{$2$}{2}-medial axis \texorpdfstring{$M_2(E)$}{M2(E)}}

For $k=2$, affine independence simply means that the two vectors are distinct. Thus the relevant property of the duality map is injectivity.

\begin{lemma}\label{lem:J-injective}
    If the norm is $C^1$ on $X\setminus\{0\}$ and strictly convex, then $J:\Sp_X\to\Sp_{X^*}$ is injective.
\end{lemma}

\begin{proof}
    Suppose that $J(u)=J(v)=f$ for $u,v\in\Sp_X$. By Lemma~\ref{lem:duality-identities}, $f(u)=f(v)=1$, and hence
    \[
        1=f\left(\frac{u+v}{2}\right)\le\norm*{\frac{u+v}{2}}\le1.
    \]
    If $u\ne v$, strict convexity gives $\norm{(u+v)/2}<1$, a contradiction. Therefore $u=v$.
\end{proof}

\begin{corollary}\label{cor:M2}
    If the norm is $C^1$ on $X\setminus\{0\}$ and strictly convex, then for every nonempty closed $E\subseteq X$,
    \[
        M_2(E)=M_2^J(E).
    \]
    In particular, $M_2(E)$ is $(n-1)$-rectifiable.
\end{corollary}

\begin{proof}
    If $x\in M_2(E)$, then $U_E(x)$ contains two distinct contact vectors, whose images under $J$ are distinct by Lemma~\ref{lem:J-injective}. Thus $x\in M_2^J(E)$. Conversely, if $x\in M_2^J(E)$, then $N_E(x)$ contains two distinct contact covectors. Since each contact covector is the image under $J$ of a contact vector, these two covectors must arise from two distinct contact vectors. Hence $x\in M_2(E)$. The rectifiability conclusion follows from Theorem~\ref{thm:intro-dual}.
\end{proof}

Strict convexity cannot be omitted, even if the norm remains $C^1$ away from the origin. We give a simple planar example. Let
\[
    B=([-1,1]\times\{0\})+\overline B_{\enorm{\cdot}}(0,1)\subseteq\R^2.
\]
The set $B$ is a centrally symmetric compact convex body with the origin in its interior, so its Minkowski functional
\[
    \norm{x}_B=\inf\{\lambda>0:x\in\lambda B\}
\]
is a norm on $\R^2$ with closed unit ball $B$. The boundary $\partial B$ is the stadium consisting of two horizontal segments joined by two semicircles. The tangent directions agree at the four junction points, so $\partial B$ is a $C^1$ embedded curve. It follows from~\cite[Proposition~13.14]{KrieglMichor1997} that $\norm{\cdot}_B$ is of class $C^1$ on $\R^2\setminus\{0\}$.

For $|r|<1$, the point $(r,-1)$ lies in the interior of the lower horizontal segment of $\partial B$. Hence
$\norm{(r,-1)}_B=1$, and its unique norming functional is $(0,-1)$, under the standard identification of $(\R^2)^*$ with $\R^2$. Therefore, $J(r,-1)=(0,-1)$ for every $|r|<1$. Now let $E=\{(-1,0),(1,0)\}$ and take $x=(a,b)$ with $b>|a|+1$. For either $e=(\pm1,0)\in E$,
\[
    e-x=b\left(\frac{\pm1-a}{b},-1\right)
    \qquad\text{and}\qquad
    \left|\frac{\pm1-a}{b}\right|<1.
\]
It follows that $\norm{e-x}_B=b$ for both points $e\in E$. Since $E$ consists of exactly these two points, they form the whole nearest-point set $\Phi_E(x)$. Therefore $U_E(x)$ consists of two distinct contact vectors, whereas
$N_E(x)=\{(0,-1)\}$. Consequently,
\[
    x\in M_2(E)\setminus M_2^J(E).
\]
Since this holds throughout the nonempty open region $\{(a,b):b>|a|+1\}$, the set $M_2(E)$ has nonempty interior and is therefore not $1$-rectifiable.

\subsection{The \texorpdfstring{$3$}{3}-medial axis \texorpdfstring{$M_3(E)$}{M3(E)}}

For $k=3$, strict convexity of both the primal and dual unit spheres is the key point. We first recall the standard duality fact needed below.

\begin{lemma}\label{lem:dual-strictly-convex}
    If the norm of $X$ is $C^1$ on $X\setminus\{0\}$, then the dual norm on $X^*$ is strictly convex.
\end{lemma}

\begin{proof}
    Suppose that $f,g\in\Sp_{X^*}$ satisfy $\dualnorm{(f+g)/2}=1$. Since $X$ is finite-dimensional, $(f+g)/2$ attains its norm at some $u\in\Sp_X$. Replacing $u$ by $-u$ if necessary, we may assume that $f(u)+g(u)=2$. Since $f(u)\le1$ and $g(u)\le1$, we obtain $f(u)=g(u)=1$. Thus $f$ and $g$ are both norming functionals of $u$. By Lemma~\ref{lem:duality-identities}, the norming functional of $u$ is unique, and hence $f=g$.
\end{proof}

\begin{corollary}\label{cor:M3}
    Let $X$ be of dimension at least two and have a norm that is $C^1$ on $X\setminus\{0\}$ and strictly convex. For every nonempty closed $E\subseteq X$,
    \[
        M_3(E)=M_3^J(E).
    \]
    Consequently, $M_3(E)$ is $(n-2)$-rectifiable.
\end{corollary}

\begin{proof}
    In a strictly convex normed space, three distinct points of the unit sphere are necessarily affinely independent. Indeed, if three distinct points were collinear, one of them would lie in the relative interior of the segment joining the other two, and strict convexity would force its norm to be strictly less than $1$.

    Thus, $x\in M_3(E)$ if and only if $U_E(x)$ contains at least three distinct contact vectors. If $x\in M_3(E)$, Lemma~\ref{lem:J-injective} shows that these vectors have three distinct images under $J$; since the dual unit sphere is strictly convex by Lemma~\ref{lem:dual-strictly-convex}, these contact covectors are affinely independent. Hence $x\in M_3^J(E)$.

    Conversely, if $x\in M_3^J(E)$, then $N_E(x)=J(U_E(x))$ contains three distinct contact covectors. They arise from three distinct contact vectors, which are affinely independent by strict convexity of the primal unit sphere. Hence $x\in M_3(E)$.
\end{proof}

\begin{corollary}\label{cor:plane}
    Let $X$ be a two-dimensional normed vector space with a norm that is $C^1$ on $X\setminus\{0\}$ and strictly convex. Let $E\subseteq X$ be nonempty and closed. Then $M_2(E)$ is $1$-rectifiable and $M_3(E)$ is countable. Thus the full conclusion of Theorem~\ref{thm:liang} remains valid.
\end{corollary}

\begin{proof}[Proof of Theorem~\ref{thm:intro-low-rank}]
    The result follows directly from Corollaries~\ref{cor:M2} and~\ref{cor:M3}; the two-dimensional conclusion is Corollary~\ref{cor:plane}.
\end{proof}

\section{Smooth examples of affine rank loss}\label{sec:smoothexamples}

Throughout this section, $\enorm{\cdot}$, $\cdot$, and $\nabla$ refer to the standard Euclidean norm, inner product, and gradient in the displayed coordinates on $\R^n$.

We treat the case $(n,k)=(3,4)$ in full detail, because it displays the rank-loss mechanism in its simplest geometric form. We prove Theorem~\ref{thm:intro-smoothexample} by constructing a strictly convex smooth norm on $\R^3$, whose unit sphere has positive Euclidean principal curvatures and is arbitrarily $C^\infty$-close to the Euclidean unit sphere, together with a four-point set $E$ whose $4$-medial axis contains a nondegenerate line segment. In the final subsection, we explain how to extend the construction to every $n\ge3$ and $4\le k\le n+1$.

The homothetic-tetrahedron mechanism underlying the construction appears in Icking, Klein, L\^e, and Ma~\cite[Section~5]{Icking1995}, who describe a continuum of such tetrahedra before modifying their construction to obtain finitely many centers stable under perturbations of the sites. Here we give a localized $C^\infty$ realization of this mechanism, retaining positive curvature and arbitrary $C^\infty$ proximity to the Euclidean sphere, and identify the corresponding loss of affine dimension of the contact covectors.

Along the resulting segment of centers, the four contact vectors are affinely independent, whereas their contact covectors have affine dimension exactly $2$. Thus the segment lies in $M_4(E)\setminus M_4^J(E)$, exhibiting a failure of the ordinary rectifiability conclusion while remaining consistent with Theorem~\ref{thm:intro-dual}.

\subsection{Construction of the unit ball}

Fix constants $0<a<b<1$ and choose a smooth function $\chi:\R\to\R$ such that
\[
    \supp\chi\subseteq(0,1),\qquad
    0\le\chi\le1,\qquad
    \chi=1\text{ on }\left[\frac a2,\frac{b+1}{2}\right],
\]
where $\supp\chi$ denotes the closed support of $\chi$. For $w=(x,y)\in\R^2$, set $r=\enorm{w}=\sqrt{x^2+y^2}$ and define
\[
    \psi(x,y)=
    \begin{cases}
        \displaystyle \chi(r)\frac{y^2}{r}&\text{if }r>0,\\[1ex]
        0&\text{if }r=0.
    \end{cases}
\]
Since $\chi$ vanishes in a neighborhood of $0$, the function $\psi$ is $C^\infty$ on $\R^2$. Moreover, $\psi$ vanishes in a neighborhood of the Euclidean unit circle. For every $\eta\in\R$ and every $(x,y)$ in the closed unit disk, set
\[
    \varphi_\eta(x,y)=-\sqrt{1-x^2-y^2}+\eta\psi(x,y),
\]
and define
\begin{equation}\label{eq:B-eta}
    B_\eta=\left\{(x,y,z)\in\R^3:x^2+y^2\le1,\ \varphi_\eta(x,y)\le z\le-\varphi_\eta(x,y)\right\}.
\end{equation}

\begin{proposition}\label{prop:unit-ball}
    For all sufficiently small $|\eta|$, the set $B_\eta$ is the closed unit ball of a strictly convex norm $\norm{\cdot}_\eta$ on $\R^3$ with the following properties:
    \begin{enumerate}[label=\textup{(\roman*)}]
        \item $\partial B_\eta$ is $C^\infty$ and has positive Euclidean principal curvatures; consequently, $\norm{\cdot}_\eta$ is $C^\infty$ on $\R^3\setminus\{0\}$;
        \item if $\rho_\eta:\Sp^2\to(0,\infty)$ is the radial function of $B_\eta$, so that $\partial B_\eta=\{\rho_\eta(\theta)\theta:\theta\in\Sp^2\}$, then $\rho_\eta\to1$ in $C^\infty(\Sp^2)$ as $\eta\to0$. Equivalently, the radial parametrizations $\theta\mapsto\rho_\eta(\theta)\theta$ converge in $C^\infty(\Sp^2;\R^3)$ to the standard inclusion $\theta\mapsto\theta$.
    \end{enumerate}
\end{proposition}

\begin{proof}
    We have $\varphi_\eta=\varphi_0+\eta\psi$, where $\varphi_0(w)=-\sqrt{1-\enorm{w}^2}$ for $w=(x,y)$ in the closed unit disk. On the one hand, we have on the open unit disk
    \[
    D^2\varphi_0(w)=\frac{1}{\sqrt{1-\enorm{w}^2}}I_2+\frac{1}{(1-\enorm{w}^2)^{3/2}}w\otimes w.
    \]
    Since $\sqrt{1-\enorm{w}^2}\le1$ and $w\otimes w$ is positive semidefinite, we have $D^2\varphi_0\ge I_2$ as quadratic forms. On the other hand, $D^2\psi$ is bounded because $\psi$ is smooth and compactly supported in the open unit disk. Hence, for all sufficiently small $|\eta|$,
    \begin{equation}\label{eq:hessian-lower-bound}
        D^2\varphi_\eta\ge\frac{I_2}{2}.
    \end{equation}
    In particular, $\varphi_\eta$ is convex on the open unit disk, and the corresponding convexity inequalities extend by continuity to the closed unit disk. Moreover, there exists $c>0$ such that $\varphi_0\le-c$ on $\supp\psi$. Since $\psi$ is bounded, shrinking $|\eta|$ further if necessary, we may therefore assume that $\varphi_\eta<0$ throughout the open unit disk.

    Since $\psi(-w)=\psi(w)$, we have $\varphi_\eta(-w)=\varphi_\eta(w)$, and the definition~\eqref{eq:B-eta} immediately gives $B_\eta=-B_\eta$. Moreover, since $\varphi_\eta$ is convex on the closed unit disk, its epigraph over the closed unit disk is convex; likewise, since $-\varphi_\eta$ is concave, its hypograph is convex. Their intersection is $B_\eta$, which is therefore convex. Since $\psi$ vanishes near the origin, $\varphi_\eta(0)=-1$, and by continuity $0\in\interior B_\eta$. Moreover, $B_\eta$ is compact. Thus $B_\eta$ is a centrally symmetric convex body containing the origin in its interior, and its Minkowski functional defines a norm $\norm{\cdot}_\eta$ on $\R^3$ having $B_\eta$ as its closed unit ball.

    We next prove~\textup{(i)}. For $\enorm{w}$ sufficiently close to $0$ or $1$, the perturbation $\eta\psi$ vanishes, so $\partial B_\eta$ agrees exactly with the Euclidean unit sphere. Elsewhere, the boundary is given by the two smooth graphs $z=\pm\varphi_\eta(w)$. Hence $\partial B_\eta$ is a $C^\infty$ submanifold of $\R^3$. Since $\partial B_\eta$ is the unit sphere of the norm $\norm{\cdot}_\eta$, Proposition~13.14 of~\cite{KrieglMichor1997} implies that $\norm{\cdot}_\eta$ is $C^\infty$ on $\R^3\setminus\{0\}$.

    It remains to prove positivity of the Euclidean principal curvatures of $\partial B_\eta$. On the lower graph parametrized by $\Gamma_\eta(w)=(w,\varphi_\eta(w))$, the outward unit normal is
    \[
        \nu(\Gamma_\eta(w))=\frac{(\nabla\varphi_\eta(w),-1)}{\sqrt{1+\enorm{\nabla\varphi_\eta(w)}^2}}.
    \]
    With our convention $\mathrm{II}_p(\xi,\zeta)=D\nu(p)[\xi]\cdot\zeta$ for $\xi,\zeta\in T_p\partial B_\eta$, we have, for $h,k\in\R^2$,
    \[
        \mathrm{II}_{\Gamma_\eta(w)}(D\Gamma_\eta(w)[h],D\Gamma_\eta(w)[k])=\frac{D^2\varphi_\eta(w)[h,k]}{\sqrt{1+\enorm{\nabla\varphi_\eta(w)}^2}}.
    \]
    Hence the estimate~\eqref{eq:hessian-lower-bound} implies that $\mathrm{II}$ is positive definite, and therefore that both principal curvatures are positive. The same conclusion holds on the upper graph by central symmetry. Finally, at points of the equator the boundary agrees locally with the Euclidean unit sphere, whose principal curvatures are both equal to $1$. Thus both Euclidean principal curvatures of $\partial B_\eta$ are positive everywhere.

    This also implies that $B_\eta$ is strictly convex. Indeed, otherwise $\partial B_\eta$ would contain a nondegenerate line segment. If $p$ is an interior point of such a segment and $\xi\ne0$ is its direction, then the curve $\gamma(t)=p+t\xi$ lies in $\partial B_\eta$ for $|t|$ sufficiently small. Since $\nu(\gamma(t))\cdot\xi=0$, differentiation at $t=0$ gives
    \[
        \mathrm{II}_p(\xi,\xi)=D\nu(p)[\xi]\cdot\xi=0,
    \]
    contradicting the positive definiteness of $\mathrm{II}$.

    Finally, we prove~\textup{(ii)}. For $\eta=0$, the set $B_0$ is the Euclidean closed unit ball, so $\rho_0\equiv1$. Since $\supp\chi\subseteq(0,1)$, there exist $0<r_0<r_1<1$ such that $\psi(w)=0$ whenever $\enorm{w}\le r_0$ or $\enorm{w}\ge r_1$. Hence $\partial B_\eta$ agrees exactly with the Euclidean unit sphere near the poles and the equator. On the compact annulus $r_0\le\enorm{w}\le r_1$, we have $\varphi_\eta-\varphi_0=\eta\psi$, so $\varphi_\eta\to\varphi_0$ in $C^\infty$ as $\eta\to0$.

    It remains to show that $\rho_\eta\to1$ in $C^\infty(\Sp^2)$. On the lower hemisphere, write $\theta=(u,\theta_3)$ with $u\in\R^2$ and $\theta_3<0$. Then $\rho_\eta(\theta)$ is characterized by $G_\eta(\rho_\eta(\theta),\theta)=0$, where $G_\eta(r,\theta)=r\theta_3-\varphi_\eta(ru)$. For $\eta=0$, we have $\partial_rG_0(1,\theta)=1/\theta_3\ne0$. Hence, by the implicit function theorem applied in local coordinates on $\Sp^2$, with $\eta$ and the local coordinates of $\theta$ as parameters and $r$ as the unknown, $(\eta,\theta)\mapsto\rho_\eta(\theta)$ is locally smooth near $\eta=0$. Since this dependence is smooth and $\rho_0\equiv1$, it follows, after restricting to compactly contained coordinate neighborhoods, that $\rho_\eta\to1$ locally in $C^\infty$ as $\eta\to0$. The same argument applies on the upper hemisphere, while near the poles and the equator one has $\rho_\eta\equiv1$. By compactness of $\Sp^2$, finitely many such neighborhoods suffice, which proves~\textup{(ii)}.
\end{proof}

\subsection{A one-parameter family of homothetic tetrahedra}

We now construct the four-point set and exhibit a one-parameter family of centers belonging to its $4$-medial axis. Assume that $|\eta|>0$ is small enough that Proposition~\ref{prop:unit-ball} applies. In the following construction, the notation $d_E$, $\Phi_E$, $U_E$, $N_E$, $J$, $M_k(E)$, and $M_k^J(E)$ refers to the norm $\norm{\cdot}_\eta$.

Set $s(t)=\sqrt{1-t^2}$ for $t\in[a,b]$, and define
\begin{align*}
    \alpha_t&=(t,0,-s(t)),& \beta_t&=(-t,0,-s(t)),\\
    \gamma_t&=(0,t,-s(t)+\eta t),& \delta_t&=(0,-t,-s(t)+\eta t).
\end{align*}
Since $[a,b]\subseteq[a/2,(b+1)/2]$, we have $\chi(t)=1$ for $t\in[a,b]$, and therefore $\varphi_\eta(\pm t,0)=-s(t)$ and $\varphi_\eta(0,\pm t)=-s(t)+\eta t$. Thus $\alpha_t,\beta_t,\gamma_t,\delta_t\in\partial B_\eta$. These four points are affinely independent, since
\begin{equation}\label{eq:determinant}
    \det(\beta_t-\alpha_t,\gamma_t-\alpha_t,\delta_t-\alpha_t)=-4\eta t^3\ne0.
\end{equation}

Fix $t_0\in(a,b)$ and set
\begin{equation}\label{eq:E}
    E=\{\alpha_{t_0},\beta_{t_0},\gamma_{t_0},\delta_{t_0}\}.
\end{equation}
For $t\in[a,b]$, let $\lambda_t=t/t_0$ and $v_t=(0,0,-s(t)+\lambda_t s(t_0))$. A direct calculation gives
\begin{equation}\label{eq:homothety}
    \{\alpha_t,\beta_t,\gamma_t,\delta_t\}=v_t+\lambda_t E.
\end{equation}
Since the four points on the left-hand side belong to $\partial B_\eta$, we have $\norm{v_t+\lambda_t e}_\eta=1$ for every $e\in E$. As $\lambda_t>0$, this is equivalent to $\norm{e-x_t}_\eta=1/\lambda_t$ for every $e\in E$, where
\begin{equation}\label{eq:center}
    x_t=-\frac{v_t}{\lambda_t}=\left(0,0,\frac{t_0}{t}s(t)-s(t_0)\right).
\end{equation}
Since $E$ consists of exactly these four points, it follows that $d_E(x_t)=1/\lambda_t$ and $\Phi_E(x_t)=E$. By~\eqref{eq:determinant} with $t=t_0$, the points of $E$ are affinely independent, and hence $x_t\in M_4(E)$ for every $t\in[a,b]$.

Finally, $(s(t)/t)'=-1/(t^2s(t))<0$. Hence the third coordinate of $x_t$ varies strictly with $t$, and the set $\{x_t:t\in[a,b]\}$ is a nondegenerate line segment on the $z$-axis.

\begin{proof}[Proof of Theorem~\ref{thm:intro-smoothexample}]
    Choose $\eta_0>0$ such that Proposition~\ref{prop:unit-ball} holds whenever $|\eta|<\eta_0$, and let $\norm{\cdot}_\eta$ be the corresponding norm with unit ball $B_\eta$. For $\eta=0$, we have $B_0$ equal to the Euclidean closed unit ball, so $\norm{\cdot}_0$ is the Euclidean norm. Proposition~\ref{prop:unit-ball} also shows that, for every $|\eta|<\eta_0$, the norm $\norm{\cdot}_\eta$ is strictly convex and $C^\infty$ away from the origin, its unit sphere has positive Euclidean principal curvatures, and $\rho_\eta\to1$ in $C^\infty(\Sp^2)$ as $\eta\to0$. Now fix $0<|\eta|<\eta_0$. With $a,b,t_0$ as above, let $E_\eta$ be the four-point set~\eqref{eq:E} corresponding to this value of $\eta$. The preceding construction shows that the nondegenerate line segment $\{x_t:t\in[a,b]\}$ is contained in $M_4(E_\eta)$. Hence $M_4(E_\eta)$ is not $0$-rectifiable.
\end{proof}

\subsection{The affine dimension of the contact covectors}

We finally verify directly that the contact covectors in the preceding example lose one affine dimension. For each $t\in[a,b]$, the four contact vectors at $x_t$ associated with the points of $E$ are precisely $\alpha_t,\beta_t,\gamma_t,\delta_t$. Indeed, by~\eqref{eq:homothety} and~\eqref{eq:center}, and since $d_E(x_t)=1/\lambda_t$, we have
\[
    \frac{e-x_t}{d_E(x_t)}=\lambda_t(e-x_t)=v_t+\lambda_t e\in\{\alpha_t,\beta_t,\gamma_t,\delta_t\}
\]
for every $e\in E$, with the corresponding choice of contact point.

Using the standard Euclidean inner product to identify $(\R^3)^*$ with $\R^3$, the covector $J(u)=D\norm{\cdot}_\eta(u)$ is represented by a normal vector to $\partial B_\eta$ at $u$, since $T_u\partial B_\eta=\ker J(u)$. Since $J(u)$ is the norming functional at $u$, this vector points outward and is uniquely determined by the normalization $J(u)(u)=1$. At a point $u=(x,y,z)$ on the lower graph $z=\varphi_\eta(x,y)$, the outward normal is proportional to
\begin{equation}\label{eq:nabla}
    \nabla\bigl(\varphi_\eta(x,y)-z\bigr)=(\partial_x\varphi_\eta,\partial_y\varphi_\eta,-1).
\end{equation}
For $t\in[a,b]$, we have $\chi(t)=1$, and in a neighborhood of the relevant points
\[
    \varphi_\eta(x,y)=-\sqrt{1-x^2-y^2}+\eta\frac{y^2}{\sqrt{x^2+y^2}}.
\]
Evaluating~\eqref{eq:nabla} at the four contact points and normalizing by $J(u)(u)=1$ gives
\begin{align*}
    J(\alpha_t)&=(t,0,-s(t)),&
    J(\beta_t)&=(-t,0,-s(t)),\\
    J(\gamma_t)&=(0,t+\eta s(t),-s(t)),&
    J(\delta_t)&=(0,-t-\eta s(t),-s(t)).
\end{align*}
Thus all four contact covectors lie in the affine plane
\[
    \{(x,y,z)\in\R^3:z=-s(t)\}.
\]
Their affine dimension is exactly $2$. Indeed, positivity of $\partial_{yy}\varphi_\eta(t,0)=1/s(t)+2\eta/t$ implies $t+\eta s(t)>t/2>0$, so $J(\alpha_t),J(\beta_t),J(\gamma_t)$ are not collinear. Consequently, $x_t\notin M_4^J(E)$ for every $t\in[a,b]$. Thus, along the line segment contained in $M_4(E)$, the corresponding contact covectors fail to have affine dimension $3$, in accordance with Theorem~\ref{thm:intro-dual}.

\subsection{Extension to higher dimensions and higher orders}

The preceding construction extends to every $n\ge3$ and $4\le k\le n+1$: there exist smooth strictly convex norms on $\R^n$ whose unit spheres have positive-definite second fundamental form and are arbitrarily $C^\infty$-close to the Euclidean unit sphere, together with a $k$-point set $E$ for which $M_k(E)$ contains a smooth embedded submanifold of dimension $n-k+2$. We outline the construction.

Set $m=k-2$ and write points of $\R^n=\R^m\times\R^{n-k+1}\times\R$ as $(u,w,z)$. Define $H(u)=u_2^2/\enorm{u}$ for $u\ne0$, where $u_2$ is the second component of $u$, and enumerate the directions $e_1,\ldots,e_m,-e_1,-e_2$ as $a_1,\ldots,a_k$, where $e_1,\ldots,e_m$ is the standard basis of $\R^m$. For every $\eta\ne0$, the lifted points $(a_i,\eta H(a_i))$ are affinely independent. Fix $t_0\in(0,1)$ and choose a smooth cutoff $\zeta(u,w)$, invariant under $(u,w)\mapsto(-u,-w)$, compactly supported in $\{(u,w):\enorm{u}^2+\enorm{w}^2<1,\ u\ne0\}$, and equal to one near each $(t_0a_i,0)$. Now extend $\zeta H$ by zero across $u=0$, replace the lower Euclidean hemisphere by
\[
    \varphi_\eta(u,w)=-\sqrt{1-\enorm{u}^2-\enorm{w}^2}+\eta\zeta(u,w)H(u),
\]
and define the body between the graphs of $\varphi_\eta$ and $-\varphi_\eta$. For all sufficiently small $|\eta|$, the same arguments as in the proof of Proposition~\ref{prop:unit-ball} show that this body is the closed unit ball of a strictly convex norm, that its boundary is $C^\infty$ with positive-definite second fundamental form, and that its radial function converges to $1$ in $C^\infty$ as $\eta\to0$.

Fix $0<|\eta|$ sufficiently small. Then, for $(t,w)$ in a neighborhood of $(t_0,0)$, we have $\zeta(ta_i,w)=1$ for every $i$, and hence the perturbed unit sphere contains the points
\[
    P_i(t,w)=\bigl(ta_i,w,-s(t,w)+\eta tH(a_i)\bigr)
    \qquad\text{where}\qquad s(t,w)=\sqrt{1-t^2-\enorm{w}^2}.
\]
Define $E=\{P_i(t_0,0):1\le i\le k\}$. Positive $1$-homogeneity of $H$ shows that these configurations are homothetic to $E$, with ratio $t/t_0$. The corresponding centers are
\[
    x_{t,w}=\left(0,-\frac{t_0}{t}w,\frac{t_0}{t}s(t,w)-s(t_0,0)\right).
\]
All points of $E$ are nearest to $x_{t,w}$, and the contact vectors are the affinely independent vectors $P_i(t,w)$. Thus $x_{t,w}\in M_k(E)$. The map $(t,w)\mapsto x_{t,w}$ has rank $n-k+2$ at $(t_0,0)$. Its image, after restricting to a sufficiently small open neighborhood, is a smooth embedded submanifold of that dimension.

We compute the corresponding contact covectors by the same normal-vector argument as above. Positive $1$-homogeneity of $H$ gives $\nabla H(ta_i)=\nabla H(a_i)$ and Euler's identity $\nabla H(a_i)\cdot a_i=H(a_i)$. Normalizing by $J(P_i(t,w))(P_i(t,w))=1$ yields
\[
    J(P_i(t,w))=\bigl(ta_i+\eta s(t,w)\nabla H(a_i),\,w,\,-s(t,w)\bigr).
\]
Since their $w$- and $z$-coordinates coincide, their affine dimension equals that of their $u$-components. Now $\nabla H(\pm e_2)=\pm e_2$, whereas $\nabla H(\pm e_1)=0$ and $\nabla H(e_j)=0$ for $3\le j\le m$. These components are therefore
\[
    te_1,\quad (t+\eta s(t,w))e_2,\quad te_3,\ldots,te_m,\quad -te_1,\quad -(t+\eta s(t,w))e_2,
\]
with the terms $te_3,\ldots,te_m$ omitted when $m=2$. Since the cutoff is identically one near $(te_1,w)$, positivity of the Hessian gives $\partial_{u_2u_2}\varphi_\eta(te_1,w)=1/s(t,w)+2\eta/t>0$, and hence $t+\eta s(t,w)>t/2>0$. The affine hull of these components contains $0$, as the midpoint of $\pm te_1$, and a nonzero multiple of every coordinate vector. It is therefore $\R^m$, so the contact covectors have affine dimension exactly $m=k-2$. The resulting submanifold is thus contained in $M_{k-1}^J(E)\setminus M_k^J(E)$.

\section{Concluding remarks and open problems}\label{sec:questions}

Theorem~\ref{thm:intro-dual} gives a natural analogue of Theorem~\ref{thm:liang} in arbitrary finite-dimensional normed spaces whose norm is $C^1$ on $X\setminus\{0\}$, with the ordinary medial axis $M_k(E)$ replaced by the dual medial axis $M_k^J(E)$, defined through the affine dimension of contact covectors rather than contact vectors. Theorem~\ref{thm:intro-low-rank} shows that, for strictly convex norms that are $C^1$ on $X\setminus\{0\}$, the two notions coincide for $k=2$ and $k=3$, so Liang's rectifiability conclusion remains valid for the ordinary medial axes $M_2(E)$ and $M_3(E)$. By contrast, the construction of Section~\ref{sec:smoothexamples} shows that the corresponding rectifiability conclusion for ordinary medial axes can fail already for $k=4$ in dimension three, even for smooth positively curved unit spheres arbitrarily $C^\infty$-close to the Euclidean sphere. The extension outlined there exhibits the same rank-loss mechanism for every $n\ge3$ and $4\le k\le n+1$.

Several natural questions remain.

\begin{question}
    Let $n\ge3$. What characterizes strictly convex norms that are $C^1$ on $X\setminus\{0\}$ and satisfy property $\mathrm{RP}_k$ for a given $4\le k\le n+1$? Are there nonellipsoidal examples? Since every affinely independent collection of unit vectors extends to an affine basis of $X$ consisting of unit vectors, property $\mathrm{RP}_{n+1}$ implies $\mathrm{RP}_k$ for every $2\le k\le n+1$. Does $\mathrm{RP}_{n+1}$ characterize norms induced by inner products?
\end{question}

\begin{question}
    Let $n\ge3$ and $4\le k\le n+1$. For a strictly convex norm that is $C^1$ on $X\setminus\{0\}$, define the rank-defect locus
    \[
        \mathcal D_k=\left\{(u_1,\ldots,u_k)\in\Sp_X^k:\dim\aff\{u_1,\ldots,u_k\}=k-1,\ \dim\aff\{J(u_1),\ldots,J(u_k)\}\le k-2\right\}.
    \]
    Does $(n-k+1)$-rectifiability of $M_k(E)$ for every nonempty closed set $E$ imply $\mathrm{RP}_k$, or equivalently $\mathcal D_k=\varnothing$? If not, which properties of $\mathcal D_k$ suffice to ensure this rectifiability conclusion? Understanding the size and geometry of $\mathcal D_k$ may help distinguish rank defects that are compatible with rectifiability from those that can generate higher-dimensional subsets of ordinary medial axes.
\end{question}

\begin{question}
    Let $n\ge3$ and $4\le k\le n+1$, and suppose that a strictly convex norm that is $C^1$ on $X\setminus\{0\}$ fails property $\mathrm{RP}_k$. Which additional geometric restrictions on a nonempty closed set $E$ ensure that $M_k(E)$ is nevertheless $(n-k+1)$-rectifiable? In particular, does this conclusion hold when $E$ is the boundary of a convex body? More generally, which geometric conditions on $E$ prevent its nearest-point configurations from meeting $\mathcal D_k$ in a way that produces higher-dimensional subsets of $M_k(E)$?
\end{question}

\begin{question}
    Suppose now that the norm is not assumed to be differentiable away from the origin. For $u\in\Sp_X$, one may replace the single contact covector $J(u)$ by the subdifferential $D^-\norm{\cdot}(u)$ introduced in Section~\ref{sec:semiconcavity}. Since the norm is convex, this agrees with the usual convex subdifferential and is the exposed face of the closed dual unit ball determined by $u$:
    \[
        D^-\norm{\cdot}(u)=\{f\in\Sp_{X^*}:f(u)=1\}.
    \]
    Given a nonempty closed set $E$ and $x\notin E$, one may therefore associate with its contact vectors the set
    \[
        \bigcup_{u\in U_E(x)}D^-\norm{\cdot}(u)\subseteq\Sp_{X^*}.
    \]
    What notion of dual rank, defined in terms of these multivalued contact covectors, yields an analogue of Theorem~\ref{thm:intro-dual}? The elementary example of the infinity norm illustrates the difficulty: at a corner of the primal unit ball, a single contact vector has an entire face of supporting covectors, whereas along a flat side, distinct contact vectors share the same supporting covector. Thus this affine dimension reflects both the nonsmoothness of the norm and the nearest-point configuration. How should these two contributions enter a rectifiability theory?
\end{question}


\begin{thebibliography}{99}

    \bibitem{Alberti1992}
    G.~Alberti, L.~Ambrosio, and P.~Cannarsa,
    \newblock On the singularities of convex functions,
    \newblock \emph{Manuscripta Math.} \textbf{76} (1992), 421--435.
    \newblock \doi{10.1007/BF02567770}

    \bibitem{AzagraLeGruyerMudarra2021}
    D.~Azagra, E.~Le Gruyer, and C.~Mudarra,
    \newblock Kirszbraun's theorem via an explicit formula,
    \newblock \emph{Canad. Math. Bull.} \textbf{64} (2021), no.~1, 142--153.
    \newblock \doi{10.4153/S0008439520000314}

    \bibitem{Blum1967}
    H.~Blum,
    \newblock A transformation for extracting new descriptors of shape,
    \newblock in W.~Wathen-Dunn (ed.), \emph{Models for the Perception of Speech and Visual Form}, MIT Press, Cambridge, MA, 1967, pp.~362--380.

    \bibitem{ChewDrysdale1985}
    L.~P. Chew and R.~L.~(Scot) Drysdale III,
    \newblock Voronoi diagrams based on convex distance functions,
    \newblock in \emph{Proceedings of the First Annual Symposium on Computational Geometry}, ACM, New York, 1985, pp.~235--244.
    \newblock \doi{10.1145/323233.323264}

    \bibitem{Choi1997}
    H.~I. Choi, S.~W. Choi, and H.~P. Moon,
    \newblock Mathematical theory of medial axis transform,
    \newblock \emph{Pacific J. Math.} \textbf{181} (1997), no.~1, 57--88.
    \newblock \doi{10.2140/pjm.1997.181.57}

    \bibitem{Erdos1945}
    P.~Erd\H{o}s,
    \newblock Some remarks on the measurability of certain sets,
    \newblock \emph{Bull. Amer. Math. Soc.} \textbf{51} (1945), 728--731.
    \newblock \doi{10.1090/S0002-9904-1945-08429-8}

    \bibitem{Erdos1946}
    P.~Erd\H{o}s,
    \newblock On the Hausdorff dimension of some sets in Euclidean space,
    \newblock \emph{Bull. Amer. Math. Soc.} \textbf{52} (1946), 107--109.
    \newblock \doi{10.1090/S0002-9904-1946-08514-6}

    \bibitem{Icking1995}
    C.~Icking, R.~Klein, N.-M.~L\^e, and L.~Ma,
    \newblock Convex distance functions in $3$-space are different,
    \newblock \emph{Fund. Inform.} \textbf{22} (1995), no.~4, 331--352.
    \newblock \doi{10.3233/FI-1995-2242}

    \bibitem{Kolasinski2023}
    S.~Kolasiński and M.~Santilli,
    \newblock Regularity of the distance function from arbitrary closed sets,
    \newblock \emph{Math. Ann.} \textbf{386} (2023), 735--777.
    \newblock \doi{10.1007/s00208-022-02407-7}

    \bibitem{KrieglMichor1997}
    A.~Kriegl and P.~W.~Michor,
    \newblock \emph{The Convenient Setting of Global Analysis},
    \newblock Mathematical Surveys and Monographs, vol.~53,
    \newblock American Mathematical Society, Providence, RI, 1997.
    \newblock \doi{10.1090/surv/053}

    \bibitem{Liang2022}
    X.~Liang,
    \newblock On the dimension of the $k$-medial axis for an arbitrary closed set,
    \newblock \emph{J. Anal. Math.} \textbf{147} (2022), 29--39.
    \newblock \doi{10.1007/s11854-022-0202-y}

    \bibitem{Lieutier2004}
    A.~Lieutier,
    \newblock Any open bounded subset of $\R^n$ has the same homotopy type as its medial axis,
    \newblock \emph{Comput.-Aided Design} \textbf{36} (2004), no.~11, 1029--1046.
    \newblock \doi{10.1016/j.cad.2004.01.011}

    \bibitem{Martini2004}
    H.~Martini and K.~J. Swanepoel,
    \newblock The geometry of Minkowski spaces---a survey. Part II,
    \newblock \emph{Expo. Math.} \textbf{22} (2004), no.~2, 93--144.
    \newblock \doi{10.1016/S0723-0869(04)80009-4}

    \bibitem{SiddiqiPizer2008}
    K.~Siddiqi and S.~M. Pizer (eds.),
    \newblock \emph{Medial Representations: Mathematics, Algorithms and Applications},
    \newblock Computational Imaging and Vision, vol.~37, Springer, Dordrecht, 2008.
    \newblock \doi{10.1007/978-1-4020-8658-8}

\end{thebibliography}
\end{document}